\documentclass{amsart}
\usepackage[a4paper,margin=45mm]{geometry}
\usepackage{amsmath,amssymb,amsthm,mathtools}
\usepackage{microtype}
\usepackage{hyperref}
\usepackage{amsrefs}

\newtheorem{theorem}{Theorem}
\newtheorem{lemma}[theorem]{Lemma}
\theoremstyle{remark}

\newcommand{\F}{\mathbb F}
\newcommand{\Z}{\mathbb Z}
\allowdisplaybreaks

\title{On symmetric and twisted-unitary modular units in group rings}
\author{Laurent Bartholdi}
\address{LB: Institut Camille Jordan, Université de Lyon 1 \textnormal{\itshape and} Section de Mathématiques, Université de Genève}
\email{laurent.bartholdi@gmail.com}
\author{Iuliya Beloshapka}
\address{IB: Google Deepmind}
\email{ibeloshapka@google.com}
\author{Roman Mikhailov}
\address{RM: St Petersburg University}
\date{September 24, 2026}

\begin{document}
\maketitle

\begin{abstract}
We study units in group rings of the Hantzsche-Wendt group. We construct a compatible system of non-trivial unitary units over $\Z/2^n\Z$ for every $n\ge1$, and ``almost-group'' units over $\Z/2\Z$, meaning non-trivial units $u$ with the property that $u-g=u^{-1}-g^{-1}$ for some group element $g$.
\end{abstract}

\section{Introduction}
Let $K$ be a field and let $G$ be a torsion-free group. Kaplansky stated in~\cite{kaplansky:problemrings} a number of problems on the group ring $KG$, including one that became the famous conjecture that $KG$ has no zero divisors. In a later note~\cite{kaplansky:problemrings2} he mentions having heard from Bergman about a symposium in Kishinev in which it was asked whether the only units in $KG$ are scalar multiples of group elements. Even though the study of units in group rings dates at least to Higman's thesis~\cite{higman:units}, this last question became known as ``Kaplansky's unit conjecture''. Higman already noted that it holds for indicable groups, namely those in which every nontrivial subgroup admits an epimorphism to $\Z$.

More generally, groups with the unique product property satisfy this conjecture~\cite{passman:gr}*{CHECK!}; so a refutation should start with a group without the unique product property. Promislow~\cite{promislow:nup} found such an example in the Hantzsche-Wendt group, the smallest crystallographic group with finite abelianization and trivial centre:
The group used throughout the paper is
\begin{equation}\label{eq:G}
G=\langle a,b\mid a^{-1}b^2a=b^{-2},\; b^{-1}a^2b=a^{-2}\rangle .
\end{equation}
It embeds in $D_\infty^3=\langle r,s\mid s^2=(rs)^2=1\rangle^3$ by
\[
  a=(r,s,s),\qquad b=(s,r,rs),
\]
and contains
\[
  \langle x=a^2, y=b^2, z=(ab)^2\rangle\cong\Z^2
\]
as a normal subgroup of index $4$.

A computer search allowed Gardam to construct a nontrivial unit in $\F_2 G$, supported on $21$ elements~\cite{gardam:units}. He later improved this to a nontrivial unit in $\Z[\zeta_8]G$, or more generally any ring containing an eightth root of unity~\cite{gardam:unitsc}. One tool that allows the search space to be restricted to a manageable size is to require \emph{more} from the unit: that it be invariant under an automorphism of $G$; or that it be \emph{unitary}, namely be invariant under the antiinvolution of $K G$ induced by $g\mapsto g^{-1}$. A weaker condition is the following: for an automorphism $\sigma\colon G\to G$ and a character $\chi\colon G\to K^\times$, a unit $u\in(K G)^\times$ is called \emph{$(\sigma,\chi)$-unitary} if
\[u=\sum u_g g,u^{-1}=\sum u_g\chi(g)\sigma(g^{-1}).\]
The units found by Gardam are all, up to translation, weakly unitary in this sense~\cite{bartholdi:units}. This reduction of the search space is all the more important since computer search (ultimately relying on SAT solving) tends to blow up exponentially.

There remain important open questions in the topic of units in $K G$. The first is whether there exists a torsion-free group with non-trivial units in $\Z G$; this is at present open, but there are units in $(\Z/p)$ for all $p$ prime~\cite{murray:units}, and we have the following result:
\begin{theorem}[Unitary units modulo all powers of $2$]\label{thm:main-2adic}
  Define the automorphism $\sigma$ of $G$ and character $\chi\colon G\to\{\pm1\}$ by
  \begin{equation}\label{eq:sigmachi}
    \sigma(a)=a,\qquad\sigma(b)=b^{-1},\qquad\chi(a)=\chi(b)=-1.
  \end{equation}
  Then there is, for all $n\ge1$, a nontrivial $(\sigma,\chi)$-unitary unit $u_n$ in $(\Z/2^n\Z)G$ such that $u_n\equiv u_{n+1}\pmod{2^n}$.

  In other words, there is a non-trivial ($\sigma,\chi)$-unitary unit $u$ in the completion $\lim(\Z/2^n\Z)G$.

  Furthermore, the unit $u_1$ is almost symmetric in the following sense: for some $g\in G$ one has
  \[u_1 + g = u_1^{-1} + g^{-1}.\]
\end{theorem}

Already obtaining a solution modulo $4$ was an open problem listed in~\cite{bartholdi:units}. We have learnt from Zane Grube and Moe Tabei that they independently obtained such solutions modulo $4$.

\subsection{Acknowledgment of AI use}
All computational work in this paper was carried out by Google DeepMind's AI Co-Mathematician~\cite{zheng+:comathematician}, which discovered the symmetric units and helped prove the results.

\section{Proof of Theorem~\ref{thm:main-2adic}}
In the group $G$ from~\eqref{eq:G}, we use the normal form
\begin{equation}\label{eq:nf}
  g=x^i y^j z^k q_s,
  \qquad q_0=1,\quad q_1=a,\quad q_2=b,\quad q_3=ab,
\end{equation}
so group elements will also be written as quadruples $(i,j,k,s)$.

Set $R=\Z G$ and let
\[\widehat R=\lim(\Z/2^n\Z)G\]
be its $2$-adic completion. (Beware that this is \emph{not} the group ring $\Z_2 G$ over the $2$-adics). For the automorphism $\sigma$ and character $\chi$ from~\eqref{eq:sigmachi}, let $\Theta$ be the anti-involution of $R G$ and $\widehat G$ induced by $g\mapsto\chi(g)\sigma(g^{-1})$.

A computer search, whose output is listed in Appendix~\ref{app:317}, produced a unit $r$ in $(\Z/8\Z)G$. It is a sum of $317$ elements, and direct multiplication, in the normal form~\eqref{eq:nf}, gives
\[
  \Theta(r)r\equiv1\pmod8.
\]
The first claim of the proof is finished by the following lifting result. Indeed, given $u\in\widehat R$ with $\Theta(u)u=u\Theta(u)=1$ we define $u_n$ as the reduction modulo $2^n$ of $u$.

\begin{lemma}\label{lem:2adic-square-root}
  Let $r\in R$ satisfy $\Theta(r)r\equiv1\pmod8$. Put
  \[
    H=\Theta(r)r,\qquad E=(H-1)/8\in R.
  \]
  Define
  \[
    S=(1+8E)^{-1/2}=\sum_{m\ge0}(-1)^m2^m\binom{2m}{m}E^m\in\widehat R,
    \qquad
    u=rS.
  \]
  Then $\Theta(u)u=\Theta(u)u=1$ so $u$ is a twisted unitary unit in $\widehat R$, and $u\equiv r\pmod2$.
\end{lemma}
\begin{proof}
The congruence assumption gives $H=1+8E$.  The binomial identity
\[
  (1+X)^{-1/2}=\sum_{m\ge0}(-1)^m\binom{2m}{m}\frac{X^m}{4^m}
\]
with $X=8E$ gives the displayed expression for $S$.  The coefficient of $E^m$ is
$(-1)^m2^m\binom{2m}{m}$, which tends to zero 2-adically, so the series converges in
$\widehat R$.

Since $\Theta(H)=H$, we have $\Theta(E)=E$ and hence $\Theta(S)=S$.  The formal identity
$(1+T)^{-1/2}(1+T)(1+T)^{-1/2}=1$ may be evaluated at $T=8E$, because only powers of the single
element $E$ occur, so $SHS=1$. Therefore
\[
        \Theta(u)u=\Theta(rS)rS=S\Theta(r)rS=SHS=1.
\]

It remains to see that $u$ is a two-sided unit. Consider the reduction $\bar r$ of $r$ modulo $2$. It is a two-sided unit in $\F_2 G$ by~\cite{aop:stablefiniteness}, since $G$ is residually amenable. Chose therefore any lift $v\in\widehat R$ of $\bar r^{-1}$. We have
\[
  uv=1+X
\]
with $X\in2\widehat R$.  The geometric series for $(1+X)^{-1}$ converges in $\widehat R$, so $uv(1+X)^{-1}=1$ and $u$ is a two-sided unit; naturally $v(1+X)^{-1}=\Theta(u)$ after the fact.
\end{proof}

For the second claim $u_1 + z = u_1^{-1} + z^{-1}$, this is again a simple matter of checking; one has $g=(abab)^{-1}$, and the $43$ terms of $u_1$ are listed in Appendix~\ref{app:43}.

Note that this element $u_1$ is dangerously close to being a zero divisor. Indeed $u=u^{-1}$ is impossible in $\F_2[G]$ without causing the existence of zero divisors, and $u^{-1}=u+g+g^{-1}$ is as close as can be in a group that satisfies the zero divisor conjecture. We hope that it may be a building lock for zero divisors in larger groups.

\appendix

\section{Support of the 43-term almost symmetric unit}
\label{app:43}

Recall that elements of $G$ are given in their normal form: $(i,j,k,s)=x^iy^jz^kq_s$, with $q_0=1$, $q_1=a$, $q_2=b$, and $q_3=ab$.

The following is a list of all elements of the support of $u_1$. One obtains $u_1^{-1}$ from $u_1$ by replacing $(0,0,-1,0)=z^{-1}$ by $(0,0,1,0)=z$.

\begin{center}
  The support of $u_1$
\footnotesize
\[\begin{array}{rl@{\qquad}rl@{\qquad}rl}
  b^{-1}& (0,-1,0,2) & 
  ab& (0,0,0,3)&
  ab^{-1}& (0,1,0,3)\\
  a^{-1}b& (-1,0,0,3)&
  a^{-1}b^{-1}& (-1,1,0,3)&
  ba& (-1,1,-1,3)\\
  ba^{-1}& (0,1,-1,3)& 
  b^{-1}a& (-1,0,-1,3)& 
  b^{-1}a^{-1}& (0,0,-1,3)\\
  aba^{-1}& (1,-1,1,2)& 
  ab^{2}& (0,-1,0,1)& 
  ab^{-1}a& (0,0,1,2)\\
  ab^{-1}a^{-1}& (1,0,1,2)& 
  a^{-2}b^{-1}& (-1,-1,0,2)& 
  a^{-1}b^{-2}& (-1,1,0,1)\\
  a^{4}& (2,0,0,0)& 
  a^{3}b& (1,0,0,3)& 
  a^{3}b^{-1}& (1,1,0,3)\\
  a^{2}ba^{-1}& (1,1,-1,3)& 
  a^{2}b^{2}& (1,1,0,0)& 
  a^{2}b^{-1}a^{-1}& (1,0,-1,3)\\
  a^{2}b^{-2}& (1,-1,0,0)& 
  aba^{-1}b& (1,0,1,0)& 
  a^{-4}& (-2,0,0,0)\\
  a^{-2}b^{2}& (-1,1,0,0)& 
  a^{-2}b^{-2}& (-1,-1,0,0)& 
  a^{-1}bab& (-1,0,1,0)\\
  baba& (0,0,-1,0)& 
  baba^{-1}& (-1,0,-1,0)& 
  ba^{-1}ba& (1,0,-1,0)\\
  a^{4}b& (2,0,0,2)& 
  a^{3}ba^{-1}& (2,-1,1,2)& 
  a^{3}b^{-2}& (1,1,0,1)\\
  a^{2}ba^{-1}b& (1,0,-1,1)& 
  ababa& (0,0,1,1)& 
  ababa^{-1}& (-1,0,1,1)\\
  a^{-3}b^{2}& (-2,-1,0,1)& 
  a^{-2}bab& (-2,0,-1,1)& 
  a^{5}b& (2,0,0,3)\\
  a^{5}b^{-1}& (2,1,0,3)& 
  a^{4}ba^{-1}& (2,1,-1,3)& 
  a^{4}b^{-1}a^{-1}& (2,0,-1,3)\\
  && a^{6}b& (3,0,0,2)& 
\end{array}\]
\end{center}

\section{The modulo 8 representative for the 2-adic lift}\label{app:317}
The following list gives the non-zero terms of the representative $r$ modulo $8$. Each term is written as a coefficient in $\{0,\ldots,7\}$ times a normal-form coordinate $(i,j,k,s)$.

\begin{center}
  The unit $r$ modulo $8$
  \footnotesize
  \begin{alignat*}{4}
    \text{coeff } & \text{coordinate} &
    \text{coeff } & \text{coordinate} &
    \text{coeff } & \text{coordinate} &
    \text{coeff } & \text{coordinate}\\
  4\;& (-4,-1,-1,1)&
  4\;& (-4,0,0,1)&
  4\;& (-3,-2,-2,1)&
  4\;& (-3,-2,-1,1)\\
  4\;& (-3,-1,-1,0)&
  4\;& (-3,-1,0,0)&
  4\;& (-3,-1,1,1)&
  6\;& (-3,0,0,0)\\
  4\;& (-3,0,0,3)&
  4\;& (-3,1,-1,3)&
  4\;& (-3,1,1,1)&
  4\;& (-2,-3,0,1)\\
  4\;& (-2,-3,1,1)&
  4\;& (-2,-2,-2,0)&
  4\;& (-2,-2,-2,1)&
  4\;& (-2,-2,-1,1)\\
  4\;& (-2,-2,0,1)&
  4\;& (-2,-2,1,0)&
  4\;& (-2,-2,1,1)&
  4\;& (-2,-2,1,2)\\
  4\;& (-2,-1,-2,0)&
  4\;& (-2,-1,-2,1)&
  4\;& (-2,-1,-2,3)&
  4\;& (-2,-1,-1,1)\\
  4\;& (-2,-1,-1,2)&
  4\;& (-2,-1,0,0)&
  5\;& (-2,-1,0,1)&
  2\;& (-2,-1,0,2)\\
  4\;& (-2,-1,0,3)&
  6\;& (-2,-1,1,0)&
  3\;& (-2,0,-1,1)&
  4\;& (-2,0,-1,2)\\
  6\;& (-2,0,-1,3)&
  7\;& (-2,0,0,0)&
  6\;& (-2,0,0,1)&
  4\;& (-2,0,0,2)\\
  4\;& (-2,0,1,0)&
  2\;& (-2,0,1,1)&
  4\;& (-2,0,1,3)&
  4\;& (-2,0,2,0)\\
  4\;& (-2,0,2,1)&
  4\;& (-2,1,-2,1)&
  2\;& (-2,1,-1,0)&
  4\;& (-2,1,-1,3)\\
  6\;& (-2,1,0,1)&
  6\;& (-2,1,0,3)&
  2\;& (-2,1,1,0)&
  2\;& (-2,1,1,1)\\
  4\;& (-2,1,1,3)&
  4\;& (-2,1,2,0)&
  4\;& (-2,2,0,0)&
  4\;& (-2,2,0,1)\\
  4\;& (-2,2,0,3)&
  4\;& (-2,2,1,0)&
  4\;& (-2,2,1,3)&
  4\;& (-2,2,2,0)\\
  4\;& (-2,2,2,1)&
  4\;& (-1,-2,-2,0)&
  4\;& (-1,-2,-1,3)&
  4\;& (-1,-2,1,0)\\
  4\;& (-1,-2,1,1)&
  6\;& (-1,-2,1,2)&
  4\;& (-1,-2,1,3)&
  4\;& (-1,-2,2,1)\\
  4\;& (-1,-2,2,2)&
  4\;& (-1,-1,-2,0)&
  4\;& (-1,-1,-2,1)&
  4\;& (-1,-1,-1,1)\\
  2\;& (-1,-1,-1,2)&
  4\;& (-1,-1,-1,3)&
  5\;& (-1,-1,0,0)&
  4\;& (-1,-1,0,1)\\
  7\;& (-1,-1,0,2)&
  4\;& (-1,-1,1,0)&
  6\;& (-1,-1,1,1)&
  4\;& (-1,-1,2,1)\\
  4\;& (-1,-1,2,2)&
  2\;& (-1,0,-2,3)&
  1\;& (-1,0,-1,0)&
  6\;& (-1,0,-1,1)\\
  2\;& (-1,0,-1,2)&
  1\;& (-1,0,-1,3)&
  4\;& (-1,0,0,1)&
  2\;& (-1,0,0,2)\\
  1\;& (-1,0,0,3)&
  5\;& (-1,0,1,0)&
  1\;& (-1,0,1,1)&
  4\;& (-1,0,1,2)\\
  2\;& (-1,0,1,3)&
  4\;& (-1,0,2,1)&
  4\;& (-1,1,-2,1)&
  4\;& (-1,1,-2,3)\\
  4\;& (-1,1,-1,1)&
  4\;& (-1,1,-1,2)&
  7\;& (-1,1,-1,3)&
  1\;& (-1,1,0,0)\\
  5\;& (-1,1,0,1)&
  7\;& (-1,1,0,3)&
  2\;& (-1,1,1,0)&
  6\;& (-1,1,1,1)\\
  4\;& (-1,1,1,2)&
  4\;& (-1,1,2,0)&
  4\;& (-1,2,-2,0)&
  4\;& (-1,2,-2,1)\\
  4\;& (-1,2,-2,3)&
  4\;& (-1,2,-1,0)&
  4\;& (-1,2,-1,1)&
  6\;& (-1,2,-1,3)\\
  4\;& (-1,2,0,1)&
  6\;& (-1,2,0,3)&
  4\;& (-1,2,1,1)&
  4\;& (-1,2,2,0)\\
  4\;& (-1,2,2,1)&
  4\;& (-1,3,-2,3)&
  4\;& (-1,3,-1,0)&
  4\;& (-1,3,-1,1)\\
  4\;& (-1,3,-1,3)&
  4\;& (-1,3,0,0)&
  4\;& (-1,3,0,1)&
  4\;& (-1,3,0,3)\\
  4\;& (-1,3,1,0)&
  4\;& (-1,3,1,1)&
  4\;& (0,-4,0,2)&
  4\;& (0,-4,1,2)\\
  4\;& (0,-3,-1,0)&
  4\;& (0,-3,-1,1)&
  4\;& (0,-3,0,0)&
  4\;& (0,-3,0,2)\\
  4\;& (0,-3,1,0)&
  4\;& (0,-3,2,2)&
  4\;& (0,-2,-2,0)&
  4\;& (0,-2,-1,0)\\
  4\;& (0,-2,-1,2)&
  4\;& (0,-2,0,2)&
  4\;& (0,-2,0,3)&
  4\;& (0,-2,1,0)\\
  4\;& (0,-2,1,1)&
  2\;& (0,-2,1,2)&
  4\;& (0,-2,1,3)&
  4\;& (0,-2,2,0)\\
  4\;& (0,-1,-2,0)&
  2\;& (0,-1,-1,1)&
  4\;& (0,-1,-1,2)&
  2\;& (0,-1,-1,3)\\
  4\;& (0,-1,0,0)&
  5\;& (0,-1,0,1)&
  7\;& (0,-1,0,2)&
  2\;& (0,-1,0,3)\\
  4\;& (0,-1,1,0)&
  4\;& (0,-1,1,1)&
  4\;& (0,-1,2,0)&
  4\;& (0,-1,2,1)\\
  4\;& (0,-1,2,2)&
  2\;& (0,0,-2,3)&
  3\;& (0,0,-1,0)&
  6\;& (0,0,-1,1)\\
  6\;& (0,0,-1,2)&
  3\;& (0,0,-1,3)&
  2\;& (0,0,0,0)&
  4\;& (0,0,0,1)\\
  2\;& (0,0,0,2)&
  1\;& (0,0,0,3)&
  6\;& (0,0,1,0)&
  5\;& (0,0,1,1)\\
  7\;& (0,0,1,2)&
  6\;& (0,0,1,3)&
  4\;& (0,0,2,1)&
  4\;& (0,0,2,2)\\
  6\;& (0,1,-1,1)&
  4\;& (0,1,-1,2)&
  1\;& (0,1,-1,3)&
  6\;& (0,1,0,0)\\
  7\;& (0,1,0,3)&
  6\;& (0,1,1,0)&
  6\;& (0,1,1,2)&
  4\;& (0,1,1,3)\\
  4\;& (0,1,2,0)&
  4\;& (0,2,-2,3)&
  2\;& (0,2,-1,3)&
  4\;& (0,2,0,0)\\
  2\;& (0,2,0,3)&
  4\;& (0,2,1,1)&
  4\;& (0,2,1,3)&
  4\;& (0,2,2,1)\\
  4\;& (0,3,-1,1)&
  4\;& (0,3,-1,3)&
  4\;& (0,3,0,0)&
  4\;& (0,3,0,3)\\
  4\;& (0,3,1,1)&
  4\;& (1,-3,-1,2)&
  4\;& (1,-3,0,2)&
  4\;& (1,-3,1,0)\\
  4\;& (1,-2,-2,0)&
  4\;& (1,-2,-2,3)&
  4\;& (1,-2,-1,0)&
  4\;& (1,-2,0,0)\\
  6\;& (1,-2,0,2)&
  4\;& (1,-2,1,0)&
  4\;& (1,-2,1,1)&
  6\;& (1,-2,1,2)\\
  4\;& (1,-2,2,0)&
  4\;& (1,-1,-2,0)&
  6\;& (1,-1,-1,0)&
  6\;& (1,-1,-1,1)\\
  6\;& (1,-1,-1,2)&
  6\;& (1,-1,-1,3)&
  7\;& (1,-1,0,0)&
  6\;& (1,-1,0,1)\\
  4\;& (1,-1,0,2)&
  6\;& (1,-1,0,3)&
  2\;& (1,-1,1,0)&
  5\;& (1,-1,1,2)\\
  4\;& (1,-1,2,2)&
  4\;& (1,0,-2,3)&
  7\;& (1,0,-1,0)&
  7\;& (1,0,-1,1)\\
  2\;& (1,0,-1,2)&
  3\;& (1,0,-1,3)&
  6\;& (1,0,0,0)&
  6\;& (1,0,0,1)\\
  2\;& (1,0,0,2)&
  7\;& (1,0,0,3)&
  3\;& (1,0,1,0)&
  2\;& (1,0,1,1)\\
  5\;& (1,0,1,2)&
  4\;& (1,0,1,3)&
  4\;& (1,0,2,1)&
  4\;& (1,0,2,2)\\
  4\;& (1,1,-2,3)&
  4\;& (1,1,-1,0)&
  2\;& (1,1,-1,1)&
  7\;& (1,1,-1,3)\\
  7\;& (1,1,0,0)&
  5\;& (1,1,0,1)&
  2\;& (1,1,0,2)&
  7\;& (1,1,0,3)\\
  2\;& (1,1,1,0)&
  6\;& (1,1,1,1)&
  4\;& (1,1,1,2)&
  4\;& (1,1,1,3)\\
  4\;& (1,1,2,0)&
  4\;& (1,2,-2,1)&
  4\;& (1,2,-2,3)&
  4\;& (1,2,-1,1)\\
  4\;& (1,2,0,3)&
  4\;& (1,2,1,2)&
  4\;& (1,3,-2,3)&
  4\;& (1,3,-1,1)\\
  4\;& (1,3,0,0)&
  4\;& (1,3,0,2)&
  4\;& (1,3,0,3)&
  4\;& (1,3,1,0)\\
  4\;& (1,3,1,1)&
  4\;& (1,3,1,2)&
  4\;& (2,-3,-1,2)&
  4\;& (2,-2,-1,0)\\
  4\;& (2,-2,-1,2)&
  4\;& (2,-1,-2,3)&
  4\;& (2,-1,-1,0)&
  4\;& (2,-1,-1,2)\\
  4\;& (2,-1,-1,3)&
  2\;& (2,-1,0,2)&
  4\;& (2,-1,1,1)&
  3\;& (2,-1,1,2)\\
  4\;& (2,-1,1,3)&
  4\;& (2,-1,2,0)&
  4\;& (2,0,-2,3)&
  4\;& (2,0,-1,2)\\
  3\;& (2,0,-1,3)&
  7\;& (2,0,0,0)&
  3\;& (2,0,0,2)&
  3\;& (2,0,0,3)\\
  2\;& (2,0,1,0)&
  4\;& (2,0,1,1)&
  2\;& (2,0,1,2)&
  4\;& (2,0,1,3)\\
  4\;& (2,0,2,0)&
  4\;& (2,1,-2,0)&
  4\;& (2,1,-1,0)&
  4\;& (2,1,-1,1)\\
  4\;& (2,1,-1,2)&
  1\;& (2,1,-1,3)&
  6\;& (2,1,0,0)&
  4\;& (2,1,0,2)\\
  1\;& (2,1,0,3)&
  4\;& (2,1,2,2)&
  4\;& (2,2,-1,0)&
  4\;& (2,2,-1,2)\\
  4\;& (2,2,0,1)&
  4\;& (2,2,0,2)&
  4\;& (2,2,0,3)&
  4\;& (2,2,1,2)\\
  4\;& (2,2,2,2)&
  4\;& (3,-2,0,2)&
  4\;& (3,-2,1,2)&
  4\;& (3,-1,-1,0)\\
  4\;& (3,-1,0,0)&
  4\;& (3,-1,1,2)&
  4\;& (3,0,-1,0)&
  4\;& (3,0,-1,1)\\
  4\;& (3,0,0,0)&
  4\;& (3,0,0,1)&
  5\;& (3,0,0,2)&
  4\;& (3,0,0,3)\\
  4\;& (3,0,1,2)&
  4\;& (3,0,2,2)&
  4\;& (3,1,-1,1)&
  4\;& (3,1,-1,2)\\
  4\;& (3,1,0,0)&
  4\;& (3,1,0,1)&
  4\;& (3,1,0,3)&
  4\;& (3,1,1,0)\\
  4\;& (3,1,1,2)&
  4\;& (4,-1,0,2)&
  4\;& (4,-1,1,2)&
  4\;& (4,0,0,2)\\
  4\;& (4,0,1,2)
    \end{alignat*}
\end{center}

\end{document}